\documentclass[11pt, a4paper]{amsart}
\usepackage[english]{babel}
\usepackage{times}
\usepackage{dsfont}
\usepackage[all]{xy}
\numberwithin{equation}{section}

\title[Isometries of pseudo-Finsler structures]{On the Lie group structure\\ of pseudo-Finsler isometries}%
\author[R. Gallego Torrom\'e, P. Piccione]{Ricardo Gallego Torrom\'e \and Paolo Piccione}
\thanks{The first author is financially supported by Fapesp, grant n.\ 2010/11934-6. The second author is partially sponsored by Fapesp
and CNPq, Brazil.}
\email{rigato39@gmail.de, piccione.p@gmail.com}
\address{Departamento de Matem\'atica\hfill\break\indent Universidade de S\~ao Paulo\hfill\break\indent Brazil}
\date{August 14th,  2013}
\begin{document}

\newtheorem{theorem}{Theorem}
\newtheorem{assertion}[theorem]{Assertion}
\newtheorem{proposition}[theorem]{Proposition}
\newtheorem{lemma}[theorem]{Lemma}
\newtheorem{definition}[theorem]{Definition}
\newtheorem{claim}[theorem]{Claim}
\newtheorem{corollary}[theorem]{Corollary}
\newtheorem{observation}[theorem]{Observation}
\newtheorem{conjecture}[theorem]{Conjecture}
\newtheorem{question}[theorem]{Question}
\newtheorem{example}[theorem]{Example}

\newtheorem*{theorem-A}{Theorem A}
\newtheorem*{theorem-B}{Theorem B (S. Deng and Z. Hou)}
\newtheorem*{corollary-n}{Corollary}

\theoremstyle{remark}\newtheorem*{remark}{Remark}
\maketitle

\begin{abstract}
Using an extension to isometries of the associated Sasaki structure, we establish a Lie transformation group structure for the set of
isometries of a pseudo-Finsler conical metric.
\end{abstract}

\begin{section}{Introduction}
A very classical problem in mathematics is to establish which topological groups have a Lie group structure (Hilbert's fifth problem).
A complete answer to this question has been given by several authors in the fifties, see for instance \cite{Glea52, Yam53}.
According to Gleason and Yamabe's result, a topological group has a Lie group structure compatible with its topology if and only if it does not contain \emph{small} subgroups, i.e., if there is some neighborhood of the identity that does not contain any proper subgroup.
A natural and important extension of this problem in Geometry, is to establish when the action of some group $G$ on a differentiable manifold
$M$ is a \emph{Lie transformation group}. Recall that a Lie transformation group consists of a Lie group $G$ and a smooth action
of $G$ on a differentiable manifold $M$ by diffeomorphisms.
By a result of Kuranishi (see \cite{Kur50}), an effective\footnote{The action of a group $G$ on a set $X$ is effective if the unique element of $G$ that fixes all the elements of $X$ is the identity.} action (by diffeomorphisms) of a \emph{locally compact} group on a smooth manifold is a Lie transformation group. This result is particularly useful
in order to establish the existence of a Lie transformation group structure for groups of distance preserving maps of metric spaces.
Namely, the isometry group of a locally compact metric space is a locally compact topological group, endowed with the
compact--open topology. By a well known result of Myers and Steenrod (see \cite{MySt39}), isometries of a Riemannian manifold
coincide with isometries of the underlying metric structure; in particular, the natural action of the isometry group of
a Riemannian manifold is a Lie transformation group. A similar argument has been employed more recently by Deng and Hou to show that
the group of isometries of a (non necessarily reversible) Finsler manifold is a Lie transformation group, see \cite{DengHou02}.
We will see here that a natural averaging procedure allows to reduce the Finsler case to the standard Riemannian case (Theorem~B).
Myers and Steenrod's result has been further developed by Palais (see \cite{palais}), who showed that the differentiable structure of
a Riemannian manifold can be recovered merely from its metric space structure.
\smallskip

When it gets to isometries of metrics with signature, like Lorentzian metrics, or more generally pseudo-Riemannian metrics,
there is no naturally associated metric space structure, and thus Myers and Steenrod's techniques do not apply.
A beautiful general theory, developed mostly in \cite{stern} and \cite{Koba}, studies the question of establishing a differentiable
structure for the set of automorphisms of a $G$-structure on a smooth manifold $M$. Such a theory allows to reduce to a relatively
simple algebraic problem the question of establishing for which Lie group $G\subset\mathrm{GL}(n)$, given any $G$-structure
$P$ on any $n$-manifold $M$, the group of automorhisms of $P$ is a (finite dimensional) Lie subgroup of the group of diffeomorphisms of $M$.
Curiously enough, such algebraic problem only involves the Lie algebra of $G$. The result applies, in particular,
to all orthogonal groups $\mathrm O(n,k)$, as well as conformal groups; thus, the set of isometries or the set of conformal
diffeomorphisms of any pseudo-Riemannian manifold is a Lie transformation group. An interesting issue of the theory is the question
of regularity of the automorphisms, and the corresponding topology in the automorphism group. Thanks to Myers--Steenrod's (or Palais')
result, for Riemannian isometries continuity is equivalent to smoothness, and all $C^k$-topologies\footnote{%
By $C^k$-topology on the group of diffeomorphisms of a manifold, we mean the \emph{weak Whitney $C^k$-topology}, i.e., the topology of uniform convergence on
compacta of all derivatives up to order $k$.}
coincide in the isometry group, $k=0,\ldots,+\infty$.
In particular,  the group of Riemannian isometries endowed with the compact--open topology is a Lie transformation group.
By the $G$-structure automorphism theory, for pseudo-Riemannian isometries one must consider the $C^1$-topology, while for conformal
diffeomorphisms one has to consider the $C^2$-topology\footnote{Although in the Riemannian case, by a somewhat involved argument, the compact--open
topology coincides with all other $C^k$-topologies in the conformal group.}. As a matter of facts,
an elementary argument using the exponential map shows that, also in the isometry group of a general pseudo-Riemannian manifold
the $C^1$-topology coincides with the compact-open topology.
Interestingly enough, the differentiability class of automorphisms
of a $G$-structure coincides with the so-called \emph{order} of the $G$-structure, which is roughly speaking the minimal order of derivatives
at a fixed point needed to determine uniquely an automorphism of the given structure. Finiteness of the order of a $G$-structure is the key property
for the development of the theory.
\smallskip

It is an important question to study automorphisms of pseudo-Finsler structures, which arise naturally in General Relativity.
A pseudo-Finsler structure\footnote{There are several notions of pseudo-Finsler structures in the literature.
Here we use a quite general notion, sometimes called \emph{conic pseudo-Finsler structure}. A somewhat different notion is
given in \cite{Beem70}. See \cite{JavSan12} for a discussion on the several notions of pseudo-Finsler manifolds.}
on a (connected) manifold $M$ consists of an open subset $\mathcal T\subset T_0M$,
where $T_0M$ denotes the tangent bundle with its zero section removed, and a smooth function $F:\mathcal T\to\mathds R^+$
satisfying the following properties:
\begin{itemize}
\item for all $p\in M$, the intersection $\mathcal T_p=\mathcal T\cap T_pM$ is a non empty open cone of the tangent space $T_pM$;\smallskip

\item $F(tv)=tF(v)$ for all $v\in\mathcal T$ and all $t>0$;
\item for all $v\in\mathcal T$, the second derivative $g_v=\left(\frac{\partial^2(F^2)}{\partial y_i\partial y_j}(v)\right)_{ij}$ in the vertical directions
is nondegerate.
\end{itemize}
By continuity, the \emph{fundamental tensor} $g_v$ has constant index, which is called the index of the pseudo-Finsler structure.
The case when $\mathcal T=T_0M$ and the index of $g_v$ is zero, i.e., $g_v$ is positive definite for all $v$, is the standard Finsler structure.
When $g_v$ does not depend on $v$, then we have a standard pseudo-Riemannian manifold.
An automorphism (or isometry)
of the pseudo-Finsler structure $(M,\mathcal T,F)$ is a diffeomorphism $f$ of $M$, with $\mathrm df(\mathcal T)=\mathcal T$ and
$F\circ\mathrm df=F$. Clearly, the set $\mathrm{Iso}(M,\mathcal T,F)$ of such automorphisms is a group with respect to composition,
and one has a natural action of $\mathrm{Iso}(M,\mathcal T,F)$ on $M$. In order to establish a Lie transformation group structure
for this set, which is the purpose of the paper, one cannot apply metric space techniques, nor $G$-structure techniques.
Namely, it is not hard to show that the $G$-structure corresponding to Finsler or pseudo-Finsler metrics has never finite order.
Similarly, also the averaging technique mentioned above for the standard Finsler example does not work in the pseudo-Finsler
case, due to the fact that:
\begin{itemize}
\item sums (or even convex combinations) of non positive definite nondegenerate symmetric bilinear forms may fail to be nondegenerate;
\item at every point $p\in M$, the indicatrix $\Sigma_p=F^{-1}(1)\cap\mathcal T_p$ is never compact.
\end{itemize}
In this paper we will use general techniques from calculus with non linear connections in vector bundles and sprays to prove the following
results.
\begin{theorem-A}
The group of isometries of a pseudo-Finsler structure $(M,F)$, endowed with the $\mathcal C^1$-topology, is a Lie transformation group
of $M$.
\end{theorem-A}
The same proof of Theorem~A will also yield the following:
\begin{corollary-n}
An isometry of a pseudo-Finsler structure $(M,F)$ is a $C^\infty$-map, and it is completely determined by its second jet at any point.
\end{corollary-n}
We will also discuss briefly the averaging technique mentioned above, that allows to reduce Deng--Hou's result to the standard
Riemannian case, proving:
\begin{theorem-B}
The group of isometries of a Finsler structure $(M,F)$, endowed with the compact--open topology, is a Lie transformation group
of $M$. Finsler isometries are smooth, and they are uniquely determined by their first jet at any point of $M$.
\end{theorem-B}
\end{section}
The proof of our results will make it clear that totally analogous results hold in the case of different notions of pseudo-Finlser structure.
More precisely, a Lie transformation group structure exists for any group of diffeomorphisms of a manifold $M$
that preserve a geodesic spray defined in suitable open subsets of $TM$, see next section for details.
\smallskip

\noindent\textbf{Aknowledgement.}\enspace The authors gratefully aknowledge the help provided by Henrique de Barros Correia Vit\'orio
during fruitful conversations on the Sasaki metric associated to a pseudo-Finsler structure.

\begin{section}{Proofs}
\subsection*{Quasi-tangent structure of $\mathbf{TM}$}
In order to define a (non linear) connection associated to a pseudo-Finsler structure, we will follow Grifone's terminology, see \cite{Grifone72}.
Let $\pi:TM\to M$ be the canonical projection; for $v\in TM$, denote by $\mathrm{Ver}_v=\mathrm{Ker}\left(\mathrm d\pi_v\right)$ the vertical subspace of $T_v(\mathcal T)$; $\mathrm{Ver}$ will denote the vertical distribution on $TM$.
First, one defines the \emph{quasi-tangent structure of $TM$} as the $(1,1)$ tensor $\mathcal J$ in $TMT$ by:
\[\mathcal J(X)=\mathfrak i_v\left(\mathrm d\pi(X)\right),\]
where $v\in T_pM$, $X\in T_v(TM)$, and $\mathfrak i_v:T_pM\to\mathrm{Ver}_v$ is the canonical identification ($\mathfrak i_v$ is the differential at $v$ of the inclusion
$T_pM\hookrightarrow TM$).
\begin{lemma}\label{thm:diffoespreserve}
If $f:M\to M$ is a diffeomorphism (of class $C^2$); the quasi-tangent structure $\mathcal J$ of $TM$ is invariant by the diffeomorphism $\mathrm df:TM\to TM$,
i.e., the pull-back $(\mathrm df)^*(\mathcal J)$ equals $\mathcal J$.
\end{lemma}
\begin{proof}
Since $\mathrm df$ send fibers of $TM$ into fibers, then clearly $\mathrm df$ preserves the vertical distribution, which is the tangent distribution to the fibers.
For $p\in M$ and $v\in\mathcal T_p$ (in fact, for $v\in T_pM$), one has the following commutative diagrams:
\begin{equation}\label{eq:2commdiagramas}
\xymatrix{T_v(TM)\ar[rr]^{\mathrm d^2f}\ar[d]_{\mathrm d\pi}&&T_{\mathrm df(v)}(TM)\ar[d]^{\mathrm d\pi}\cr T_pM\ar[rr]_{\mathrm df}&&T_{f(p)}M}
\qquad \xymatrix{T_pM\ar[rr]^{\mathrm df}\ar[d]_{\mathfrak i_v}&&T_{f(p)}M\ar[d]^{\mathfrak i_{\mathrm df(v)}}\cr \mathrm{Ver}_v\ar[rr]_{\mathrm d^2f}&&\mathrm{Ver}_{\mathrm d^2f}}.
\end{equation}
The commutativity of the first diagram is obvious. For the second, it suffices to differentiate the commutative diagram:
\[\xymatrix{T_pM\ar[rr]^{\mathrm df(p)}\ar[d]_{\text{inclusion}}&& T_{f(p)}M\ar[d]^{\text{inclusion}}\cr TM\ar[rr]_{\mathrm df}&&TM.}\]
The equality $(\mathrm df)^*(\mathcal J)=\mathcal J$ follows readily from \eqref{eq:2commdiagramas}.
\end{proof}
\subsection*{Orthogonal distribution associated to a pseudo-Finsler structure}
Consider now a pseudo-Fisler structure $(M,\mathcal T,F)$, and let $S$ denote the vector field in $\mathcal T$ given by the geodesic spray
of $F$.
There is a complement to this space associated to $S$, the \emph{horizontal} space, which is defined as follows.
 The spray $S$ satisfies the identity\footnote{%
The identity $\mathcal J(S)=C$ means that the integral curves of $S$ are of the form $t\mapsto\gamma'(t)\in \mathcal T$, for some curve $t\mapsto\gamma(t)\in M$.
Such curves $\gamma$ are precisely the geodesics of $S$.
The identity $[C,S]=S$ means that affine reparameterizations of geodesics of $S$ are geodesics.} $\mathcal J(S)=C$, where $C$ is the \emph{tautological vertical field of $TM$}, or \emph{Liouville field}, (i.e., $C_v=\mathfrak i_v(v)$), and the identity $[C,S]=S$, where $[\cdot,\cdot]$ are the Lie brackets of $TM$.
Moreover, the Lie derivative $\Gamma_S=-\mathcal L_S(\mathcal J)$ of the quasi-tangent structure $\mathcal J$ is a $(1,1)$ tensor on $\mathcal T$ that satisfies (see \cite{Grifone72}):
\begin{equation}\label{eq:LiederJ}
(\Gamma_S)^2=\mathrm{Id},\qquad \mathrm{Ker}\left(\Gamma_S+\mathrm I\right)=\mathrm{Ver}.
\end{equation}
By \eqref{eq:LiederJ}, $\mathrm{Hor}^S:=\mathrm{Ker}\left(\Gamma_S-\mathrm{Id}\right)$ is a distribution in $\mathcal T$ which is complementary to $\mathrm{Ver}$,
and it will be called the \emph{orthogonal distribution associated to the pseudo-Finsler structure $(M,\mathcal T,F)$}.
\subsection*{The Sasaki metric}
Denote by $k$ the index of the fundamental tensor $g_v$ of the pseudo-Finsler structure $(M,\mathcal T,F)$; we will now define a pseudo-Riemannian metric $g^F$ on $\mathcal T$ having index
$2k$. For $v\in\mathcal T$, the spaces $\mathrm{Ver}_v$ and $\mathrm{Hor}^S_v$ are $g^F$-orthogonal. The restriction of $g^F$ to $\mathrm{Ver}_v$ is
the push-forward of $g_v$ by the isomorphism $\mathfrak i_v:T_pM\to\mathrm{Ver}_v$, and the restriction of $g^F$ to $\mathrm{Hor}^S_v$ is
the pull-back of $g_v$ by the isomorphism $\mathrm d\pi_v:\mathrm{Hor}^S_v\to T_pM$. Clearly, $g^F$ is a smooth $(0,2)$-tensor field on $\mathcal T$
which is everywhere nondegenerate and of index $2k$; the tensor $g^F$ is the \emph{Sasaki metric} of the pseudo-Finlser structure $(M,\mathcal T,F)$.

The central result is the following:
\begin{proposition}\label{thm:FinslSasisometries}
Let $f:M\to M$ be an isometry of $(M,\mathcal T,F)$, i.e., $f:M\to M$ is a diffeomorphism of class $C^2$, with $\mathrm df(\mathcal T)=\mathcal T$, and
$(\mathrm df)^*(F)=F$. Then, $\mathrm df:\mathcal T\to\mathcal T$ is an isometry of the Sasaki metric $g^F$.
\end{proposition}
\begin{proof}
We have already observed in Lemma~\ref{thm:diffoespreserve} that $\mathrm df$ preserves the vertical distribution $\mathrm{Ver}$ and the
quasi-tangent structure $\mathcal J$. Since $f$ is an isometry of $(M,\mathcal T,F)$, then $\mathrm df$ preserves the geodesic spray $S$.
Hence, by construction, $\mathrm df$ preserves also the orthogonal distribution $\mathrm{Hor}^S$.
The commutativity of the diagram on the left of \eqref{eq:2commdiagramas} (when $\mathrm d^2f$ is restricted to $\mathrm{Hor}_v^S$) shows that
$\mathrm d^2f$ preserves the restriction of $g^F$ to the horizontal distribution. The commutativity of the diagram on the right of \eqref{eq:2commdiagramas}
shows that $\mathrm d^2f$ preserves the restriction of $g^F$ to the vertical distribution. In conclusion, $\mathrm df$ is an isometry of
the pseudo-Riemannian manifold $(\mathcal T,g^F)$.
\end{proof}
\begin{remark}
It is also immediate to prove that, conversely, if $f:M\to M$ is a $C^2$-maps such that $\mathrm df(\mathcal T)=\mathcal T$ and such that
$\mathrm df\vert_{\mathcal T}:\mathcal T\to\mathcal T$ is a $g^F$-isometry, then infact $f$ is a diffeomorphism of $M$ and it is
an isometry of $(M,\mathcal T,F)$. For this, one uses the fact that $\mathcal T_p=\mathcal T\cap T_pM$ is a nonempty open subset of $T_pM$ for all $p\in M$.
\end{remark}
\subsection*{Final argument}
Let us denote by $\mathrm{Iso}(M,\mathcal T,F)$ the group of $\mathcal C^2$-isometries of the pseudo-Finsler structure $(M,\mathcal T,F)$, and by $\mathrm{Iso}(\mathcal T,g^F)$
the isometry group of the pseudo-Riemannian manifold $(\mathcal T,g^F)$. It is well known (see for instance \cite{Koba}) that $\mathrm{Iso}(\mathcal T,g^F)$ is a Lie group,
and that the natural action of $\mathrm{Iso}(\mathcal T,g^F)$ on $\mathcal T$ is smooth. Moreover, every element of $\mathrm{Iso}(\mathcal T,g^F)$
is determined by its first jet at any point of $\mathcal T$.

The proof of Theorem~A and its Corollary will be obtained directly from the following two results.
\begin{proposition}\label{thm:main}
The map $\mathrm{Iso}(M,\mathcal T,F)\ni f\mapsto\mathrm df\in\mathrm{Iso}(\mathcal T,g^F)$ is an injective group homomorphism,
whose image is closed in the $C^1$-topology.
\end{proposition}
\begin{proof}
The given map is a group homomorphism, by the chain rule; it is obviously injective.
In order to prove that its image is closed in the $C^1$-topology, assume that $f_n$ is a sequence of $C^2$-diffeomorphisms of $M$, with $\mathrm df_n(\mathcal T)=\mathcal T$
such that $(\mathrm df_n)\big\vert_{\mathcal T}$ converges as $n\to\infty$ in the $C^1$-topology to a $C^1$-diffeomorphism $\Psi:\mathcal T\to\mathcal T$, then,
by elementary arguments:
\begin{itemize}
\item[(a)] $f_n$ is $C^1$-convergent to some diffeomorphism $f_\infty$ of $M$ (namely, if $s$ is a local section of $TM$ taking values in $\mathcal T$,
then locally $f_n=\pi\circ(\mathrm df_n)\circ s$);
\item[(b)] $\Psi=\mathrm df_\infty$, and thus $f_\infty$ is of class $C^2$.
\end{itemize}
This concludes the proof.
\end{proof}
The statements in Theorem~A and its Corollary follow almost entirely from Proposition~\ref{thm:main}.
As to the action of $\mathrm{Iso}(M,\mathcal T,F)$, what Proposition~\ref{thm:main} says is that the map
$\mathrm{Iso}(M,\mathcal T,F)\times\mathcal T\ni(f,v)\mapsto\mathrm df(v)\in\mathcal T$ is smooth, and it is
a Lie transformation group of $\mathcal T$. From this, it follows easily (see Lemma~\ref{thm:Lietransfgroupfb} below)
that the natural action of $\mathrm{Iso}(M,\mathcal T,F)$ on $M$ is a Lie transformation group of $M$.
\subsection*{Lie transformation groups of submersions}
Assume that $q:E\to B$ is a smooth surjective submersion, and let $f:E\to E$ be a diffeomorphism
that carries fibers of $q$ (diffeomorphically) onto fibers. Then one has an induced map $\widetilde f:B\to B$, which is again
a diffeomorphism. If $G$ is a Lie transformation group of $E$ such that the action of every element $g\in G$, $E\ni x\mapsto g\cdot x\in E$,
carries fibers onto fibers, then one has an induced action of $G$ on the base $B$. We state the following elementary result, which may have some interest of its own.
\begin{lemma}\label{thm:Lietransfgroupfb}
Let $q:E\to B$ be a smooth surjective submersion, and let $G$ be a Lie transformation group of the total space $E$.
Assume that the action of each element of $G$ carries fibers of $q$ onto fibers. Then, the induced action of
$G$ on the base $B$ makes $G$ into a Lie transformation group of $B$.
\end{lemma}
\begin{proof}
The smoothness of the induced action of $G$ on $B$ follows easily from the existence of local sections of $q$.
\end{proof}
The proof of Theorem~A is concluded by applying Lemma~\ref{thm:Lietransfgroupfb} to the surjective submersion $\pi\vert_\mathcal T:\mathcal T\to M$ and
to the Lie transformation group $\mathrm{Iso}(M,\mathcal T,F)\times\mathcal T\ni(f,v)\mapsto\mathrm df(v)\in\mathcal T$.
\subsection*{The Finsler case: proof of Theorem~B}
Given a Finsler structure $(M,F)$, one can define a Riemannian metric $h_F$ on $M$ obtained as the average of the fundamental tensor.
More precisely, for all $p\in M$, let $\Sigma_p$ be the indicatrix of $F$ at $p$:
\[\Sigma_p=\big\{v\in T_pM : F(v)=1\big\}.\]
Then, $h_F$ is defined by:
\begin{equation}
h_F(v,w)=\int_{\Sigma_p} g_u(v,w)\,\mathrm d\Omega_p(u),
\end{equation}
where $v,w\in T_pM$, and $\mathrm d\Omega_p$ is the volume associated to the Riemannian metric in $\Sigma_p$ given by the restriction
of the fundamental tensor. This averaged metric was first defined in \cite{ricardo}.

The proof of Theorem~B is obtained readily from the following:
\begin{proposition}
The group of isometries of $(M,F)$ is contained as a closed subgroup of $\mathrm{Iso}(M,h_F)$ (in the compact--open topology).
\end{proposition}
\begin{proof}
If $f:M\to M$ is a diffeomorphism that preserves $F$, then $\mathrm df$ carries indicatrices onto indicatrices, and it also preserves
the fundamental tensor of $F$, as well as the volume forms on the indicatrices associated to the fundamental tensor. Thus, $f$ preserves $h_F$.
The condition $f^*(F)=F$ is closed in the $C^1$-topology, hence $\mathrm{Iso}(M,F)$ is closed in $\mathrm{Iso}(M,h_F)$ with respect to the
$C^1$-topology. On the other hand, the compact--open topology and the $C^1$-topology coincide on  $\mathrm{Iso}(M,h_F)$.
This concludes the proof.
\end{proof}
\end{section}

\end{document}